\RequirePackage{fix-cm}
\documentclass[smallextended]{svjour3}       
\smartqed  
\usepackage{graphicx}
\usepackage{mathptmx}      

\usepackage{amsmath}
\usepackage{amssymb}

\usepackage{booktabs}

\newcommand{\fejer}{Fej\'er}
\newcommand{\frechet}{Fr\'echet}
\newcommand{\scal}[2]{\left\langle{#1},{#2}  \right\rangle}
\newcommand{\menge}[2]{\big\{{#1}~\big |~{#2}\big\}}
\newcommand{\Fix}{\ensuremath{\operatorname{Fix}}}
\newcommand{\ran}{\ensuremath{\operatorname{ran}}}
\newcommand{\inte}{\ensuremath{\operatorname{int}\,}}
\newcommand{\nnn}{\ensuremath{{n\in{\mathbb N}}}}
\newcommand{\RR}{\ensuremath{\mathbb R}}
\newcommand{\NN}{\ensuremath{\mathbb N}}
\newcommand{\pinf}{\ensuremath{+\infty}}
\newcommand{\RPP}{\ensuremath{\mathbb R}_{++}}
\newcommand{\RP}{\ensuremath{\mathbb R}_{+}}
\newcommand{\ball}{\ensuremath{\mathrm{ball}}}
\newcommand{\sgn}{\ensuremath{\operatorname{sgn}}}
\newcommand{\Id}{\ensuremath{\operatorname{Id}}}
\newcommand{\To}{\ensuremath{\rightrightarrows}}
\newcommand{\exi}{\ensuremath{\exists\,}}

\spnewtheorem{fact}{Fact}{\bf}{\it}

\journalname{Journal of Optimization Theory and Applications}
\begin{document}

\title{On the Finite Convergence of a Projected Cutter
Method\thanks{HHB was partially supported by a Discovery Grant
and an Accelerator Supplement of the Natural Sciences and
Engineering Research Council of Canada (NSERC) and by the Canada
Research Chair Program.  CW was partially supported by a grant
from Shanghai Municipal Commission for Science and Technology
(13ZR1455500).  XW was partially supported by a Discovery Grant
of NSERC.  JX was partially supported by NSERC grants of HHB and
XW.}} 


\author{Heinz~H.~Bauschke \and
        Caifang~Wang \and
	Xianfu~Wang \and
	Jia~Xu
}


\institute{Heinz H.\ Bauschke (corresponding author) \at
              Mathematics, University of British
Columbia, Kelowna, B.C.\ V1V~1V7, Canada\\
              \email{heinz.bauschke@ubc.ca}           
           \and
	   Caifang Wang \at
	   Department of Mathematics,
Shanghai Maritime University, China\\
           \email{cfwang@shmtu.edu.cn}           
	   \and
	   Xianfu Wang \at
              Mathematics, University of British
Columbia, Kelowna, B.C.\ V1V~1V7, Canada\\
           \email{shawn.wang@ubc.ca}           
	   \and
	   Jia Xu \at 
           Mathematics, University of British
Columbia, Kelowna, B.C.\ V1V~1V7, Canada\\
           \email{jia.xu@ubc.ca}           
}

\date{Received: date / Accepted: date}

\maketitle

\begin{abstract}
The subgradient projection iteration is a classical method for solving
a convex inequality. Motivated by works of Polyak and of Crombez,
we present and analyze a more general
method for finding a fixed point of a cutter, 
provided that the fixed point set has nonempty interior. 
Our assumptions on the parameters are more general than existing
ones.  Various limiting examples and comparisons are provided.
\end{abstract}

\section{Introduction}

Throughout this paper,
we assume that
\begin{equation}
\text{$X$ is a real Hilbert space}
\end{equation}
with inner product $\scal{\cdot}{\cdot}$ and induced norm
$\|\cdot\|$. 
We also assume that 
$T \colon X\to X$ is a \emph{cutter},
i.e., 
$\Fix T := \menge{y\in X}{y=Ty}\neq\varnothing$ and 
that furthermore
$(\forall x\in X)(\forall y\in \Fix T)$
$\scal{y-Tx}{x-Tx}\leq 0$;
equivalently,
\begin{equation}
(\forall x\in X)(\forall y\in \Fix T)\quad
\|Tx-y\|^2 + \|x-Tx\|^2 \leq \|x-y\|^2. 
\end{equation}
Cutters are also known as \emph{quasi firmly nonexpansive operators}.
We also assume that 
$C$ is a closed and convex subset of $X$
such that
$C \cap \Fix T \neq\varnothing$.
Our aim is to 
\begin{equation}
\label{e:0424a}
\text{find a point in $C\cap\Fix T\neq\varnothing$.}
\end{equation}
Because $T$ can be a subgradient projector 
(see Example~\ref{ex:0424b} below), 
\eqref{e:0424a} is quite flexible and includes the problem
of solving convex inequalities. 
For further information on cutters and subgradient projectors,
we refer the reader to 
\cite{bb96,MOR,BWWX1,Cegielski,CL,CS,CenZen,Comb93,Comb97,CombLuo,Pauwels,Poljak,Polyakbook,PolyakHaifa,SY,YO1,YO2,YSY,YY} 
and the references therein. 

Given $r\geq 0$, we follow Crombez \cite{Crombez} and define the operator 
$U_r\colon X\to X$ at $x\in X$ by 
\begin{equation}
\label{e:Ur}
U_rx := 
\begin{cases}
\displaystyle x + \frac{r+\|Tx-x\|}{\|Tx-x\|}(Tx-x) = Tx +
\frac{r}{\|Tx-x\|}(Tx-x), &\text{if $x\neq Tx$;}\\
x, &\text{otherwise.}
\end{cases}
\end{equation}
When $T$ is a subgradient projector,
then $U_r$ was also studied by Polyak \cite{PolyakHaifa}. 
Note that $\Fix U_r = \Fix T$. 

Our goal is to solve \eqref{e:0424a} algorithmically via
sequence $(x_n)_\nnn$ generated by $x_0\in X$ and
\begin{equation}
(\forall \nnn) 
\quad
x_{n+1} := P_{C}U_{r_n}x_n,
\end{equation}
where $P_C$ is the projector\footnote{$P_C$ is the unique
operator from $X$ to $C$ satisfying $(\forall x\in X)(\forall c\in C)$
$\|x-P_Cx\|\leq\|x-c\|$.} onto $C$ 
and the sequence of parameters $(r_n)_\nnn$ lying in $\RPP :=
\menge{\xi\in\RR}{\xi>0}$ satisfies a 
divergent-series condition.

\emph{
We will obtain \emph{finite convergence results} for this and more general
algorithms provided some constraint
qualification is satisfied. In the present setting,
our results complement and extend 
results by Crombez for cutters and by Polyak
for subgradient projectors. 
}

The paper is organized as follows.
In Section~\ref{s:aux}, we collect various
auxiliary results, that will facilitate the presentation
of the main results in Section~\ref{s:main}.
Limiting examples are presented in Section~\ref{s:limex}.
In Section~\ref{s:compare}, we compare to existing results.
Future research directions are discussed in
Section~\ref{s:persp}. Finally, Section~\ref{s:conc} concludes the
paper. 
Notation is standard and follows e.g., \cite{BC2011}. 

\section{Auxiliary Results}

\label{s:aux}

\subsection{Cutters}

We start with the most important instance of a cutter,
namely Polyak's subgradient projector \cite{Poljak}. 

\begin{example}[subgradient projector]
\label{ex:0424b}
Let $f\colon X\to\RR$ be convex and continuous such that $\menge{x\in
X}{f(x)\leq 0}\neq\varnothing$,
and let $s\colon X\to X$ be a selection of $\partial f$,
i.e., $(\forall x\in X)$ $s(x)\in\partial f(x)$. 
Then the \emph{associated subgradient projector}, 
defined by 
\begin{equation}
(\forall x\in X)\quad 
G_fx := 
\begin{cases}
\displaystyle x - \frac{f(x)}{\|s(x)\|^2}s(x), &\text{if $f(x)>0$;}\\
x, &\text{otherwise,}
\end{cases}
\end{equation}
is a cutter. 
\end{example}

We now collect some inequalities and identities that
will facilitate the proofs of the main results.
The inequality
$\|U_rx-y\|^2 \leq \|Tx-y\|^2 - r^2$,
which is a consequence of \ref{l:0424e2} in the next lemma,
was also observed by Crombez in \cite[Lemma~2.3]{Crombez}. 

\begin{lemma}
\label{l:0424e}
Let $y\in \Fix T$, let $r\in\RPP$, and
suppose that $\ball(y;r)\subseteq \Fix T$ and that
$x\in X\smallsetminus \Fix T$. 
Set
\begin{equation}
\tau_x := \scal{x-y}{(x-Tx)/\|x-Tx\|} - \big(r + \|x-Tx\|\big).
\end{equation}
Then the following hold:
\begin{enumerate}
\item 
\label{l:0424e1}
$\tau_x \geq 0$. 
\item 
\label{l:0424e2}
$\|U_rx-y\|^2 = \|Tx-y\|^2 - r^2 -2r\tau_x \leq \|Tx-y\|^2 - r^2$. 
\item 
\label{l:0424e3}
$\|U_rx-y\|^2 = \|x-y\|^2 - (r+\|x-Tx\|)^2 -2\tau_x(r+\|x-Tx\|)
\leq \|x-y\|^2 - (r+\|x-Tx\|)^2 \leq \|x-y\|^2 - r^2 - \|x-Tx\|^2$. 
\end{enumerate}
\end{lemma}
\begin{proof}
\ref{l:0424e1}:
Set $z := y + r(x-Tx)/\|x-Tx\|$.
Then $z\in \ball(y;r)\subseteq \Fix T$. 
Since $T$ is a cutter, we obtain
\begin{subequations}
\begin{align}
0 &\geq \scal{z-Tx}{x-Tx}\\
&=\scal{y+r(x-Tx)/\|x-Tx\|-Tx}{x-Tx}\\
&=\scal{y-Tx}{x-Tx} + r\|x-Tx\|\\
&=\scal{y-x}{x-Tx} + \|x-Tx\|^2 + r\|x-Tx\|.
\end{align}
\end{subequations}
Rearranging and dividing by $\|x-Tx\|$ yields 
$\scal{x-y}{(x-Tx)/\|x-Tx\|} \geq r+\|x-Tx\|$ and
hence $\tau_x\geq 0$.

\ref{l:0424e2}:
Using \eqref{e:Ur}, we derive the identity from
\begin{subequations}
\begin{align}
\|U_rx-y\|^2 &= \big\| x+ (\|x-Tx\|+r)/\|x-Tx\|(Tx-x)-y\big\|^2\\
&= \big\|(Tx-y) + r (Tx-x)/\|Tx-x\|\big\|^2\\
&=\|Tx-y\|^2 + r^2 + 2r\scal{(Tx-x)+(x-y)}{(Tx-x)/\|Tx-x\|}\\
&=\|Tx-y\|^2 + r^2 +2r\|x-Tx\| - 2r\scal{x-y}{(x-Tx)/\|x-Tx\|}\\
&=\|Tx-y\|^2 - r^2 - 2r\tau_x. 
\end{align}
\end{subequations}
The inequality follows immediately from \ref{l:0424e1}. 

\ref{l:0424e3}:
Using \ref{l:0424e2}, we obtain
\begin{subequations}
\begin{align}
\|U_rx-y\|^2 &= \|(x-y)+(Tx-x)\|^2 - r^2 - 2r\tau_x\\
&= \|x-y\|^2 + \|x-Tx\|^2 + 2\scal{x-y}{Tx-x} - r^2 - 2r\tau_x\\
&= \|x-y\|^2 - \|x-Tx\|^2 - 2(\tau_x+r)\|x-Tx\|-r^2-2r\tau_x\\
&=  \|x-y\|^2 - (r+\|x-Tx\|)^2 -2\tau_x(r+\|x-Tx\|). 
\end{align}
\end{subequations}
The inequalities now follow from \ref{l:0424e1}. 
\qed
\end{proof}

We note in passing that $U_r$ itself is not necessarily a cutter:

\begin{example}[$U_r$ need not be a cutter]
\label{ex:0425a}
Suppose that $X=\RR$ and that 
$T$ is the subgradient projector associated with the function
$f\colon\RR\to\RR\colon x\mapsto x^2-1$.
Then $\Fix T = [-1,1]$. 
Let $r\in\RP := \menge{\xi\in\RR}{\xi\geq 0}$.
Then 
\begin{equation}
(\forall x\in \RR\smallsetminus\Fix T)
\quad
U_rx  = \frac{x}{2} + \frac{1}{2x} - r\sgn(x).
\end{equation}
Choosing $y:=1\in\Fix T$ and $x:=y+\varepsilon\notin\Fix T$, where
$\varepsilon\in\RPP$, 
we may check that $U_r$ is not a cutter\footnote{In fact, $U_r$ is not even
a relaxed cutter in the sense of \cite[Definition~2.1.30]{Cegielski}.}
 when $\varepsilon$ is sufficiently small and $r>0$. 
\end{example}

We now obtain the following result
concerning a relaxed version\footnote{$U_{r,\eta}$ can also be called a generalized
relaxation of $T$ with relaxation parameter $\eta$; see \cite[Definition~2.4.1]{Cegielski} .} of $U_r$.
Item~\ref{c:0424f5} also follows from \cite[Corollary~2.4.3]{Cegielski}. 

\begin{corollary}
\label{c:0424f}
Let $y\in \Fix T$, let $r\in\RPP$, let $\eta\in\RP$,
and suppose that $\ball(y;r)\subseteq \Fix T$ and that
$x\in X\smallsetminus \Fix T$. 
Set
\begin{equation}
U_{r,\eta}x := 
x + \eta \frac{r+\|x-Tx\|}{\|Tx-x\|}(Tx-x). 
\end{equation}
Then the following hold\footnote{We note that 
item~\ref{c:0424f4} can also be deduced from \cite[(2.27)]{Cegielski} with $\lambda
= (r+\|x-Tx\|)/\|x-Tx\|$, $z=y$, and $\delta=r$ in 
\cite[Proposition~2.1.41]{Cegielski}. 
This observation, as well as a similar one for \ref{c:0424f5}, is due to a referee.}:
\begin{enumerate}
\item
\label{c:0424f1}
$U_{r,\eta}x = (1-\eta)x + \eta U_r x$. 
\item
\label{c:0424f2}
$\|U_{r,\eta}x-y\|^2 = 
\eta\|U_{r}x-y\|^2 + (1-\eta)\|x-y\|^2 - \eta(1-\eta)\|x-U_{r}x\|^2$. 
\item
\label{c:0424f3}
$\|U_rx-x\| = r + \|x-Tx\|$. 
\item
\label{c:0424f4}
$\|U_rx-y\|^2 \leq \|x-y\|^2 - (r + \|x-Tx\|)^2 
= \|x-y\|^2 - \|x-U_{r}x\|^2$. 
\item
\label{c:0424f5}
$\|U_{r,\eta}x-y\|^2 \leq \|x-y\|^2 - \eta(2-\eta)(r + \|x-Tx\|)^2
= \|x-y\|^2 - \eta^{-1}(2-\eta)\|x-U_{r,\eta}x\|^2$. 
\end{enumerate}
\end{corollary}
\begin{proof}
\ref{c:0424f1}: This is a simple verification. 

\ref{c:0424f2}: 
Using \ref{c:0424f1}, we obtain
$\|U_{r,\eta}x-y\|^2 = \|(1-\eta)(x-y)+\eta(U_{r}x-y)\|^2$.
Now use \cite[Corollary~2.14]{BC2011} to obtain the identity. 

\ref{c:0424f3}:
This is immediate from \eqref{e:Ur}. 

\ref{c:0424f4}:
Combine \ref{c:0424f3} with Lemma~\ref{l:0424e}\ref{l:0424e3}. 

\ref{c:0424f5}:
Combine \ref{c:0424f1}--\ref{c:0424f4}.
\qed
\end{proof}

\subsection{Quasi Projectors}

\begin{definition}[quasi projector]
$Q\colon X\to X$ is a 
\emph{quasi projector} of $C$ 
if $\ran Q = \Fix Q = C$ and
$(\forall x\in X)(\forall c\in C)$
$\|Qx-c\|\leq\|x-c\|$. 
\end{definition}

\begin{example}[projectors are quasi projectors]
$P_C$ is a quasi projector of $C$.
More generally\footnote{This observation is a due to a referee.}, 
if $R\colon X\to X$ is quasi nonexpansive, i.e.,
$(\forall x\in X)(\forall y\in \Fix R)$ $\|Rx-y\|\leq\|x-y\|$ and
$C\subseteq \Fix R$, then $P_C\circ R$ is a quasi projector of
$C$. 
\end{example}

It can be shown (see \cite[Proposition~3.4.4]{Thesis}) 
that when $C$ is an affine subspace, then 
the only quasi projector of $C$ is the projector.
However, we will now see that for certain cones there
are quasi projectors different from projectors. 

\begin{proposition}[reflector of an obtuse cone]
\label{p:0424c}
{\rm (See \cite[Lemma~2.1]{BK}.)}
Suppose that 
$C$ is an obtuse cone, i.e., 
$\RP C = C$ and 
$C^\ominus := \menge{x\in X}{\sup\scal{C}{x}=0} \subseteq - C$. 
Then the reflector $R_C := 2P_C-\Id$ is nonexpansive and 
$\ran R_C = \Fix C = C$.
\end{proposition}

\begin{corollary}
\label{c:0424d}
Suppose that $C$ is an obtuse cone and let
$\lambda \colon X \to [1,2]$. 
Then 
\begin{equation}
Q\colon X\to X\colon x\mapsto \big(1-\lambda(x)\big)x + \lambda(x) P_Cx
\end{equation}
is a quasi projector of $C$. 
\end{corollary}
\begin{proof}
Since, for every $x\in X$,
we have $Q(x)\in [P_Cx,R_Cx]$ and the result thus follows from
Proposition~\ref{p:0424c}.
\qed
\end{proof}

\begin{example}
Suppose $X=\RR^d$ and $C=\RP^d$.
Then $R_C$ is a quasi projector.
\end{example}
\begin{proof}
Because $C^\ominus = -C$,
this follows from Corollary~\ref{c:0424d} with $\lambda(x)\equiv 2$. 
\qed
\end{proof}

\begin{remark}
A quasi projector need not be 
continuous because we may choose $\lambda$ in Proposition~\ref{p:0424c}
discontinuously. 
\end{remark}

\subsection{\fejer\ Monotone Sequences}

Recall that a sequence $(x_n)_\nnn$ in $X$ is \fejer\ monotone with
respect to a nonempty subset $S$ of $X$ if
\begin{equation}
(\forall s\in S)(\forall\nnn)\quad
\|x_{n+1}-s\|\leq\|x_n-s\|.
\end{equation}
Clearly, every \fejer\ monotone sequence is bounded.

We will require the following key result. 

\begin{fact}[Raik]
\label{f:Raik}
Let $(x_n)_\nnn$ be a sequence in $X$ that
is \fejer\ monotone with respect to a 
subset $S$ of $X$.
If $\inte S\neq\varnothing$, then
$(x_n)_\nnn$ converges strongly to some point in $X$ and
$\sum_\nnn\|x_n-x_{n+1}\| <\pinf$.
\end{fact}

\begin{proof}
See \cite{Raik} or e.g.\ \cite[Proposition~5.10]{BC2011}. 
\qed
\end{proof}

\subsection{Differentiability}

\begin{lemma}
\label{l:0426a}
Suppose that $X$ is finite-dimensional, 
let $f\colon X\to\RR$ be convex and \frechet\ differentiable such
that $\inf f(X)<0$. 
Then for every $\rho\in\RPP$, we have
\begin{equation}
\inf \menge{\|\nabla f(x)\|}{x\in \ball(0;\rho)\cap
f^{-1}(\RPP)}>0.
\end{equation}
\end{lemma}
\begin{proof}
Let $\rho\in\RPP$ and assume to the contrary that the conclusion
fails.
Then there exists a sequence $(x_n)_\nnn$ in 
$\ball(0;\rho)\cap f^{-1}(\RPP)$ and a point $x\in
\ball(0;\rho)$ such that $x_n\to x$
and $\nabla f(x_n)\to 0$. 
It follows that $f(x)\geq 0$ and $\nabla f(x)=0$, which is
clearly 
absurd. 
\qed
\end{proof}

\section{Finitely Convergent Cutter Methods}

\label{s:main}

From now on, we assume that
\begin{subequations}
\begin{equation}
\label{e:onr}
\text{$(r_n)_\nnn$ is a sequence in $\RPP$ such that $r_n\to 0$,}
\end{equation}
that
\begin{equation}
\text{$(\eta_n)_\nnn$ is a sequence in $\left]0,2\right]$,}
\end{equation}
and that
\begin{equation}
\text{$Q_C$ is a quasi projector of $C$.}
\end{equation}
\end{subequations}
We further assume that
$x_0\in C$ and that
$(x_n)_\nnn$ is generated by
\begin{equation}
\label{e:seq}
(\forall\nnn)\quad
x_{n+1} :=
\begin{cases}
 Q_C\big(x_n+\eta_n(U_{r_n}x_n-x_n)\big), &\text{if $x_n\notin \Fix T$;}\\
x_n, &\text{otherwise.}
\end{cases}
\end{equation}
Note that $(x_n)_\nnn$ lies in $C$. 
Also observe that if $x_n$ lies in $\Fix T$, then so does $x_{n+1}$. 

We are now ready for our first main result.

\begin{theorem}
\label{t:main1}
Suppose that $\inte(C\cap \Fix T)\neq\varnothing$ and
that $\sum_{\nnn}\eta_nr_n=\pinf$.
Then $(x_n)_\nnn$ lies eventually in $C\cap\Fix T$. 
\end{theorem}
\begin{proof}
We argue by contradiction.
If the conclusion is false, then \emph{no} term of the sequence in $(x_n)_\nnn$
lies in $\Fix T$, i.e., $(x_n)_\nnn$ lies in $X\smallsetminus \Fix T$. 
By assumption, there exist 
$z\in C\cap\Fix T$ and
$r\in\RPP$ and such that
$\ball(z;2r)\subseteq C\cap\Fix T$. 
Hence
\begin{equation}
\label{e:0424g}
\big(\forall y\in \ball(z;r)\big)\quad
\ball(y;r)\subseteq C\cap \Fix T.
\end{equation}
Since $r_n\to 0$, there exists $m\in\NN$ such
that $n\geq m$ implies $r_n\leq r$. 
Now let $n\geq m$ and $y\in\ball(z;r)$.
Using the assumption that $Q_C$ is a quasi projector of $C$,
that $y\in C$, \eqref{e:0424g}, and Corollary~\ref{c:0424f}, we obtain
\begin{subequations}
\begin{align}
\|x_{n+1}-y\| &= \big\|Q_C\big(x_n+\eta_n(U_{r_n}x_n-x_n)\big)-y\big\|\\
&\leq \|x_n+\eta_n(U_{r_n}x_n-x_n)-y\|\\
&\leq \|x_n-y\|. 
\end{align}
\end{subequations}
Hence the sequence
\begin{equation}
\big(x_m,x_m+\eta_m(U_{r_m}x_m-x_m),x_{m+1},x_{m+1}+\eta_{m+1}(U_{r_{m+1}}x_{m+1}-x_{m+1}),x_{m+2},\ldots\big)
\end{equation}
is \fejer\ monotone with respect to $\ball(z;r)$. 
It follows from Fact~\ref{f:Raik} and Corollary~\ref{c:0424f}\ref{c:0424f3} that 
\begin{equation}
\pinf>
\sum_{n\geq m} \eta_n\|x_n-U_{r_n}x_n\| = 
\sum_{n\geq m} \eta_n\big(r_n+\|x_n-Tx_n\|\big) \geq 
\sum_{n\geq m}\eta_nr_n, 
\end{equation}
which is absurd because $\sum_\nnn \eta_nr_n=\pinf$. 
\qed
\end{proof}

We now present our second main result.
Compared to Theorem~\ref{t:main1}, we have a less restrictive
assumption on $(\Fix T,C)$ but a more restrictive one 
on the parameters $(r_n,\eta_n)$. 
The proof of Theorem~\ref{t:main2} is more or less implicit in
the works by Crombez \cite{Crombez} and Polyak
\cite{PolyakHaifa}; see Remark~\ref{r:Polyak} and
Remark~\ref{r:Crombez}. 

\begin{theorem}
\label{t:main2}
Suppose that $C\cap\inte \Fix T\neq\varnothing$ and that 
$\sum_\nnn \eta_n(2-\eta_n)r_n^2=\pinf$. 
Then $(x_n)_\nnn$ lies eventually in $C\cap\Fix T$. 
\end{theorem}
\begin{proof}
Similarly to the proof of Theorem~\ref{t:main1},
we argue by contradiction and assume the conclusion is false.
Then $(x_n)_\nnn$ must lie in $X\smallsetminus \Fix T$. 
By assumption, there exist 
$y\in \Fix T$ and $r\in\RPP$ such that
$\ball(y;r)\subseteq \Fix T$. 
Because $r_n\to 0$, there exists $m\in\NN$ such that
$n\geq m$ implies $r_n\leq r$.
Let $n\geq m$.
Using also the assumption that $Q_C$ is a quasi projector of $C$
and Corollary~\ref{c:0424f}\ref{c:0424f5}, 
we deduce that 
\begin{subequations}
\begin{align}
\|x_{n+1}-y\|^2 &=
\big\|Q_C\big(x_n+\eta_n(U_{r_n}x_n-x_n)\big)-y\big\|^2\\
&\leq \|x_n+\eta_n(U_{r_n}x_n-x_n)-y\|^2\\
&\leq \|x_n-y\|^2 -
\eta_n(2-\eta_n)\big(r_n+\|x_n-Tx_n\|\big)^2\\
&\leq \|x_n-y\|^2 - \eta_n(2-\eta_n)r_n^2.
\end{align}
\end{subequations}
This implies
\begin{equation}
\|x_m-y\|^2 \geq \sum_{n\geq m}\big(\|x_n-y\|^2 -
\|x_{n+1}-y\|^2\big) \geq 
\sum_{n\geq m} \eta_n(2-\eta_n)r_n^2 = \pinf,
\end{equation}
which contradicts our assumption on the parameters.
\qed
\end{proof}

Theorem~\ref{t:main1} and Theorem~\ref{t:main2} have various applications.
Since every resolvent of a maximally monotone operator is firmly
nonexpansive and hence a cutter, we obtain the following result.

\begin{corollary}
\label{c:Shawn}
Let $A\colon X\To X$ be maximally monotone, suppose that
$Q_C=P_C$, that $T= (\Id+A)^{-1}$, and
that one of following holds:
\begin{enumerate}
\item $\inte(C\cap A^{-1}0)\neq\varnothing$
and $\sum_\nnn \eta_nr_n=\pinf$.
\item $C\cap \inte A^{-1}0 \neq\varnothing$
and $\sum_\nnn \eta_n(2-\eta_n)r_n^2=\pinf$.
\end{enumerate}
Then $(x_n)_\nnn$ lies eventually in $C\cap A^{-1}0$. 
\end{corollary}

Corollary~\ref{c:Shawn} applies in particular
to finding a constrained critical point of a 
convex function. When specializing further 
to a normal cone operator, we obtain the following result.

\begin{example}[convex feasibility]
Let $D$ be a nonempty closed convex subset of $X$,
and suppose that $Q_C = P_C$, that $T = P_D$, and that one of
the following holds:
\begin{enumerate}
\item 
$\inte(C\cap D)\neq\varnothing$ and $\sum_\nnn r_n=\pinf$.
\item
$C\cap\inte D\neq\varnothing$ and $\sum_\nnn r_n^2 = \pinf$.
\end{enumerate}
Then the sequence $(x_n)_\nnn$, generated by
\begin{equation}
(\forall \nnn)\quad
x_{n+1} := P_C\bigg(P_Dx_n + r_n\frac{P_Dx_n-x_n}{\|P_Dx_n-x_n\|} \bigg)
\end{equation}
if $x_n\notin D$ and $x_{n+1}:=x_n$ if $x_n\in D$, 
lies eventually in $C\cap D$.
\end{example}

\begin{remark}[relationship to Polyak's work]
\label{r:Polyak}
In \cite{PolyakHaifa}, 
B.T.\ Polyak considers random algorithms for solving constrained
systems of convex inequalities. Suppose that only one consistent
constrained convex inequality is
considered. Hence the cutters used are all subgradient
projectors (see Example~\ref{ex:0424b}). 
Then his algorithm coincides with the one considered
in this section and thus is comparable.
We note that our Theorem~\ref{t:main1} is more flexible because
Polyak requires $\sum_\nnn r_n^2=\pinf$ (see
\cite[Theorem~1 and Section~4.2]{PolyakHaifa}) provided that
$0<\inf_\nnn \eta_n \leq \sup_\nnn \eta_n < 2$ while
we require only $\sum_\nnn r_n=\pinf$ in this case. 
Regarding our Theorem~\ref{t:main2}, we note that
our proof essentially follows his proof which actually works
for cutters --- not just subgradient projectors --- and under a
less restrictive constraint qualification. 
\end{remark}

\begin{remark}[relationship to Crombez's work]
\label{r:Crombez}
In \cite{Crombez}, G.\ Crombez considers asynchronous parallel
algorithms for finding a point in the intersection of the fixed
point sets of finitely
many cutters --- without the constraint set $C$. 
Again, we consider the case when we are dealing with only one
cutter. Then Crombez's convergence result (see
\cite[Theorem~2.7]{Crombez}) is similar to Theorem~\ref{t:main2};
however, he requires that 
the radius $r$ of some ball contained in $\Fix T$ be known which
may not always be realistic in practical applications. 
\end{remark}

We will continue our comparison in
Section~\ref{s:compare}. 
While it is not too difficult to extend Theorem~\ref{t:main1} and
Theorem~\ref{t:main2} to deal with finitely many cutters, we have 
opted here for simplicity rather than maximal generality.
Instead, we focus in the next section 
on limiting examples. 

We conclude this section with a comment on the proximal point algorithm.

\begin{remark}[proximal point algorithm]
Suppose that $A$ is a maximally monotone operator on $X$ 
(see, e.g., \cite{BC2011} for relevant background information)
such that $Z := A^{-1}0\neq\varnothing$.
Then its resolvent $J_A := (\Id+A)^{-1}$ is firmly nonexpansive
--- hence a cutter --- with $\Fix J_A = Z$. 
Let $y_0\in X$ and set $(\forall \nnn)$ 
$y_{n+1} := J_Ay_n$. Then $(y_n)_\nnn$,
the sequence generated by the proximal point algorithm, 
converges weakly to a point in $Z$.
If
\begin{equation}
\label{e:0429a}
(\exi \bar{x}\in X)\quad
0\in\inte A\bar{x},
\end{equation}
then the convergence is finite
(see \cite[Theorem~3]{Rockprox}).
On the other hand, our algorithms impose that
$\inte \Fix T \neq\varnothing$, i.e.,
\begin{equation}
\label{e:0429b}
(\exi \bar{x}\in X)\quad
\bar{x}\in\inte A^{-1}0. 
\end{equation}
(Note that \eqref{e:0429a} and \eqref{e:0429b} are
independent: 
If $A$ is $\partial \|\cdot\|$, then
$0\in\inte A0$ yet $\inte A^{-1}0=\varnothing$.
And if $A = \nabla d^2_{\ball(0;1)}$,
then $0\in \inte A^{-1}0$ while $A = 2(\Id-P_{\ball(0;1)})$ is single-valued.)
\end{remark}

\section{Limiting Examples}

\label{s:limex}

In this section, we collect several examples that
illustrate the boundaries of the theory.

We start by showing that the conclusion of Theorem~\ref{t:main1} and
Theorem~\ref{t:main2} both may fail to hold if the divergent-series
condition is not satisfied. 

\begin{example}[divergent-series condition is important]
Suppose that $X=C=\RR$, 
that $f\colon \RR\to\RR\colon x\mapsto x^2-1$,
and that $T=G_f$
is the subgradient projector associated with $f$.
Suppose that $x_0>1$,
set $r_{-1} := x_0-1>0$ and 
$(\forall\nnn)$ 
$r_{n} := r_{n-1}^2/(4(1+r_{n-1}))$. 
Then $(r_n)_\nnn$ lies in $\RPP$,
$r_n\to 0$, and
$\sum_{\nnn} r_n <\pinf$ and
hence $\sum_\nnn r_n^2 <\pinf$.
However, the sequence $(x_n)_\nnn$ 
generated by \eqref{e:seq} lies in 
$\left]1,\pinf\right[$ and hence
does not converge finitely to a point in $\Fix T = [-1,1]$.
Furthermore, the classical subgradient projector
iteration $(\forall\nnn)$ $y_{n+1}=Ty_n$ converges
to some point in $\Fix T$, but not finitely when $y_0\notin \Fix T$. 
\end{example}
\begin{proof}
It is clear that $\Fix T= [-1,1]$.
Observe that 
$(\forall\nnn)$
$0<r_n \leq ({1}/{4})r_{n-1}\leq
(1/4)^{n+1} r_{-1}$. 
It follows that $r_n\to 0$ and that
$\sum_\nnn r_n$ and $\sum_\nnn r_n^2$ are
both convergent series. 
Now suppose that $r_{n-1} = x_n-1 > 0$ for some 
$\nnn$. 
It then follows from
Example~\ref{ex:0425a} that
\begin{equation}
x_{n+1} = \frac{x_n}{2} + \frac{1}{2x_n} - r_n
= \frac{(x_n-1)^2}{2x_n} + 1 - r_n
= \frac{r_{n-1}^2}{2(1+r_{n-1})} + 1 - r_n
= r_n+1.
\end{equation}
Hence, by induction,
$(\forall\nnn)$ $x_{n} = 1+r_{n-1}$ and
therefore $x_n\to 1^+$. 

As for the sequence $(y_n)_\nnn$, it is follows from
Polyak's seminal work (see \cite{Poljak}) that 
$(y_n)_\nnn$ converges to some point in $\Fix T$. 
However, by e.g.\ \cite[Proposition~9.9]{BWWX1},
$(y_n)_\nnn$ lies outside $\Fix T$ whenever $y_0$ does. 
\qed
\end{proof}

The next example illustrates that we cannot
expect finite convergence if the interior
of $\Fix T$ is empty, in the context
of Theorem~\ref{t:main1} and Theorem~\ref{t:main2}. 

\begin{example}[nonempty-interior condition is important]
Suppose that $X=C=\RR$, 
that $f\colon \RR\to\RR\colon x\mapsto x^2$,
and that $T=G_f$ is the subgradient projector
associated with $f$. 
Then $\Fix T=\{0\}$ and hence
$\inte\Fix T = \varnothing$. 
Set $x_0:=1/2$,
and set $(\forall\nnn)$
$w_n := (n+1)^{-1/2}$
and 
$r_n = w_n$ if $U_{w_n}x_n\neq 0$
and $r_n=2w_n$ if $U_{w_n}x_n=0$.
Then $r_n\to 0$ and $\sum_\nnn r_n^2=\pinf$.
The sequence $(x_n)_\nnn$ generated
by \eqref{e:seq} converges to $0$ but not finitely. 
\end{example}
\begin{proof}
The statements concerning $(r_n)_\nnn$ are clear.
It follows readily from the definition that
$(\forall x\in\RR)(\forall r\in\RP)$
$Tx= x/2$ and $U_rx = x/2 - r\sgn(x)$. 
Since $x_0=1/2$, $w_0=1$,
$U_1x_0 = -3/4\neq 0$, and $r_0=w_0=1$,
it follows that $0< |x_0/2| < r_0$.
We now show that for every $\nnn$, 
\begin{equation}
\label{e:0425b}
0 < |x_n/2| < r_n.
\end{equation}
This is clear for $n=0$. Now assume
\eqref{e:0425b} holds for some $\nnn$.

\emph{Case~1:} $|x_n| = 2w_n$.\\
Then $U_{w_n}x_n = x_n/2 - \sgn(x_n)w_n = 0$.
Hence $r_n=2w_n$ and thus
$x_{n+1} = U_{r_n}x_n  
= x_n/2 - 2w_n\sgn(x_n)
= \sgn(x_n)w_n - 2w_n\sgn(x_n)
= -\sgn(x_n)w_n$. 
Thus
$0 < |x_{n+1}/2| = w_n/2 = 1/(2\sqrt{n+1}) < 1/\sqrt{n+2}=w_{n+1}\leq
r_{n+1}$, which yields \eqref{e:0425b} with $n$ replaced by $n+1$. 

\emph{Case~2:} $|x_n| \neq 2w_n$.\\
Then $U_{w_n}x_n = x_n/2 - \sgn(x_n)w_n \neq 0$.
Hence $r_n = w_n$ and thus 
$x_{n+1} = U_{r_n}x_n  = x_n/2 - r_n\sgn(x_n)$.
It follows that 
$|x_{n+1}| = r_n - |x_n/2|>0$.
Hence $0<|x_{n+1}/2|$ and
also 
$|x_{n+1}| < r_n = w_n < 2w_{n+1}\leq 2r_{n+1}$.
Again, this is \eqref{e:0425b} with $n$ replaced by $n+1$.

It follows now by induction that \eqref{e:0425b} holds for every $\nnn$.
\qed
\end{proof}

We now illustrate that when $\Fix T=\varnothing$, then
$(x_n)_\nnn$ may fail to converge. 

\begin{example}
Suppose that $X=C=\RR$,
that $f\colon \RR\to\RR\colon x\mapsto x^2+1$,
and that $T=G_f$ is the subgradient projector associated with $f$.
Let $y_0\in\RR$ and suppose that
$(\forall\nnn)$ $y_{n+1}:=Ty_n$.
Then $(y_n)_\nnn$ is either not well defined or it diverges.
Suppose that $x_0 > 1/\sqrt{3}$,
set $k_0 := x_0 - 1/\sqrt{3}>0$ and
$(\forall\nnn)$ $k_{n+1} := \sqrt{(n+1)/(n+2)}k_n$.
Suppose that 
\begin{equation}
(\forall\nnn)\quad 
r_n := \frac{1}{2}\bigg( \sqrt{3} + 2k_{n+1} + k_n -
\frac{1}{k_n+1/\sqrt{3}}\bigg).
\end{equation}
Then $r_n\to 0^+$ and $\sum_\nnn r_n^2 = \pinf$.
Moreover, the sequence $(x_n)_\nnn$ generated by 
\eqref{e:seq} diverges. 
\end{example}
\begin{proof}
Clearly, $\Fix T = \varnothing$ and
one checks that 
\begin{equation}
\label{e:0425c}
(\forall r\in\RP)(\forall x\in\RR\smallsetminus\{0\})\quad
U_rx = \frac{x}{2} - \frac{1}{2x} - r\sgn(x).
\end{equation}
If some $y_n=0$, then the sequence $(y_n)_\nnn$ is not well defined.

\emph{Case~1}: $(\exi\nnn)$ $y_n=1/\sqrt{3}$.\\
Then $x_{n+1}= Tx_n = U_0x_n = x_n/2 - 1/(2x_n) = -1/\sqrt{3}
= -x_n$ and similarly $x_{n+2}=-x_{n+1}=x_n$.
Hence the sequence eventually oscillates between 
$1/\sqrt{3}$ and $-1/\sqrt{3}$. 

\emph{Case~2}: $(\exi\nnn)$ $|y_n|=1$.\\
Then $y_{n+1} = 0$ and the sequence is not well defined. 

\emph{Case~3}: $(\forall\nnn)$ $|y_n|\notin\{1,1/\sqrt{3}\}$.\\
Using the Arithmetic Mean--Geometric Mean inequality, we obtain
\begin{equation}
|y_{n+1}-y_n| = \left|\frac{y_n}{2}-\frac{1}{2y_n}-y_n\right|
= \frac{1}{2}\left| y_n + \frac{1}{y_n}\right|
= \frac{1}{2}\left( |y_n| + \frac{1}{|y_n|}\right) \geq 1
\end{equation}
for every $\nnn$. 
Therefore, $(y_n)_\nnn$ is divergent or not well defined. 

We now turn to the sequence $(x_n)_\nnn$. 
Observe that $0 < k_n = \sqrt{n/(n+1)}k_{n-1} = \cdots =
k_0/\sqrt{n+1} \to 0^+$ and hence $(k_n)_\nnn$ is strictly
decreasing. 
It follows that $r_n\to 0^+$ and
that $r_n > (2k_{n+1}+k_n)/2 > 3k_{n+1}/2 = 3k_0/(2\sqrt{n+2})$.
Thus, $\sum_\nnn r_n^2 = \pinf$. 
Next, \eqref{e:0425c} yields
\begin{subequations}
\begin{align}
x_1 &= \frac{x_0}{2} - \frac{1}{2x_0} - r_0 \\
&= \frac{k_0+1/\sqrt{3}}{2} - \frac{1}{2\big(k_0+1/\sqrt{3})} -
\frac{1}{2} \bigg( \sqrt{3} + 2k_{1} + k_0 -
\frac{1}{k_0+1/\sqrt{3}}\bigg)\\
&= -\frac{1}{\sqrt{3}} - k_1.
\end{align}
\end{subequations}
Hence $x_1<0$ and we then see analogously that $x_2 =
1/\sqrt{3}+k_2 > 0$.
We inductively obtain
\begin{equation}
(\forall\nnn)\quad
0<x_{2n} = \frac{1}{\sqrt{3}} + k_{2n}
\;\;\text{and}\;\;
0>x_{2n+1} = -\frac{1}{\sqrt{3}} - k_{2n+1}.
\end{equation}
It follows that $(-1)^nx_n \to 1/\sqrt{3}$;
therefore, $(x_n)_\nnn$ is divergent. 
\qed
\end{proof}

\section{Comparison}

\label{s:compare}

In this section, we assume for notational simplicity\footnote{If we replace
\frechet\ differentiability by mere continuity, then we may
consider a selection of the subdifferential operator $\partial f$
instead.} that 
\begin{equation}
\text{$f\colon X\to\RR$ is convex and \frechet\ differentiable
with $\menge{x\in X}{f(x)\leq 0}\neq\varnothing$}
\end{equation}
and that 
\begin{equation}
T=G_f\colon X\to X\colon x\mapsto 
\begin{cases}
\displaystyle x - \frac{f(x)}{\|\nabla f(x)\|^2}\nabla f(x), &\text{if $f(x)>0$;}\\
x, &\text{otherwise}
\end{cases}
\end{equation}
is the associated subgradient projector (see
Example~\ref{ex:0424b}). 
Then \eqref{e:Ur} turns into 
\begin{equation}
U_rx = 
\begin{cases}
\displaystyle x - \frac{f(x)+r\|\nabla f(x)\|}{\|\nabla f(x)\|^2}
\nabla f(x), &\text{if $f(x)>0$;}\\
x, &\text{otherwise}
\end{cases}
\end{equation}
and \eqref{e:seq} into 
\begin{equation}
\label{e:ours}
(\forall\nnn)\quad
x_{n+1} =
\begin{cases}
 Q_C\bigg(x_n-\eta_n\displaystyle\frac{f(x_n)+r_n\|\nabla
 f(x_n)\|}{\|\nabla f(x_n)\|^2}\nabla
 f(x_n)\bigg), &\text{if $f(x_n)>0$;}\\
x_n, &\text{otherwise.}
\end{cases}
\end{equation}

In the algorithmic setting of Section~\ref{s:main},
Polyak uses $\eta \equiv\eta_n\in \left]0,2,\right[$
(e.g.\ $\eta=1.8$; see \cite[Section~4.3]{PolyakHaifa}). 
In the present setting, his framework requires $\sum_\nnn
r_n^2=\pinf$. 

When $C=X$, one also has 
the following similar yet different update formula
\begin{equation}
\label{e:ccp}
(\forall\nnn)\quad
y_{n+1} =
\begin{cases}
y_n-\eta_n\displaystyle\frac{f(y_n)+\varepsilon_n}{\|\nabla f(y_n)\|^2}\nabla
 f(y_n), &\text{if $f(y_n)>0$;}\\
y_n, &\text{otherwise,}
\end{cases}
\end{equation}
where $0<\inf_\nnn \eta_n \leq \sup_\nnn \eta_n<2$ and 
$(\varepsilon_n)_\nnn$ is a strictly decreasing sequence in $\RPP$ 
with $\sum_\nnn\varepsilon_n=\pinf$. 
In this setting, this is also known as the 
\emph{Modified Cyclic Subgradient Projection Algorithm (MCSPA)},
which finds its historical roots in works by 
Fukushima \cite{Fuku}, by De Pierro and Iusem \cite{DPI}, and
by Censor and Lent \cite{CL};
see also \cite{CCP,IusMol86,IusMol87,Pang} for related works. 
Note that MCSPA requires the existence of a \emph{Slater point}, i.e.,
$\inf f(X)<0$, which is more restrictive than our assumptions
(consider, e.g., the squared distance to the unit ball). 
Let us now link the assumption on the parameters of the MCSPA \eqref{e:ccp} 
to \eqref{e:ours}. 

\begin{proposition}
Suppose that $X=C$ is finite-dimensional,
that $\inf f(X)<0$,
that $\eta_n\equiv 1$, 
that $\sum_\nnn r_n=\pinf$ (recall \eqref{e:onr}), and
that $(\forall\nnn)$ $\varepsilon_n = r_n\|\nabla
f(x_n)\| > 0$. 
Then $\varepsilon_n \to 0$ and $\sum_\nnn\varepsilon_n=\pinf$.
\end{proposition}
\begin{proof}
Corollary~\ref{c:0424f}\ref{c:0424f4} implies 
that $(x_n)_\nnn$ is bounded. 
Because $\nabla f$ is continuous, we obtain
that $\sigma := \sup_\nnn \|\nabla f(x_n)\|<\pinf$. 
By Lemma~\ref{l:0426a}, there
exists $\alpha\in\RPP$ such that
if $f(x_n)>0$, then $\|\nabla f(x_n)\|\geq \alpha$. 
Hence
\begin{equation}
(\forall\nnn)\quad
f(x_n)>0
\;\;\Rightarrow\;\;
0<\alpha r_n \leq \|\nabla f(x_n)\|r_n = \varepsilon_n \leq
\sigma r_n,
\end{equation}
and therefore $\sum_\nnn \varepsilon_n=\pinf$. 
\qed
\end{proof}

The next example shows that our assumptions are independent of
those on the MCSPA. 

\begin{example}
Suppose that $X=C=\RR$, that
$f\colon \RR\to\RR\colon x\mapsto x^2-1$,
that $r_n = (n+1)^{-1}$ if $n$ is even and
$r_n = n^{-1/2}$ if $n$ is odd,
and that $\eta_n\equiv 1$. 
Clearly, $r_n\to 0$ and $\sum_\nnn r_n^2=\pinf$. 
However, $(\varepsilon_n)_\nnn := (r_n|f'(x_n)|)_\nnn$
is not strictly decreasing.
\end{example}
\begin{proof}
The sequence $(x_n)_\nnn$ is bounded. 
Suppose that $f(x_n)>0$ for some $\nnn$.
By Example~\ref{ex:0425a},
\begin{equation}
\label{e:0426b}
x_{n+1} = U_{r_n}x_n = \frac{x_n}{2} + \frac{1}{2x_n} -
r_n\sgn(x_n).
\end{equation}
Assume that $n$ is even, say $n=2m$, where $m\geq 2$,
and that $1<x_{2m} < (2m+1)/2$.
Then $x_{2m} > 2x_{2m}/\sqrt{2m+1}$ and 
\begin{equation}
\varepsilon_{2m} =
r_{2m}|f'(x_{2m})|
= 2r_{2m}x_{2m}
= \frac{2x_{2m}}{2m+1}.
\end{equation}
Hence, using \eqref{e:0426b}, 
\begin{equation}
x_{2m+1} = \frac{x_{2m}}{2} + \frac{1}{2x_{2m}} - r_{2m}
> \frac{x_{2m}}{2} + \frac{1}{2m+1} - \frac{1}{2m+1} =
 \frac{x_{2m}}{2},
\end{equation}
and therefore
\begin{equation}
2x_{2m+1} > x_{2m} > \frac{2x_{2m}}{\sqrt{2m+1}}.
\end{equation}
Thus
$\varepsilon_{2m+1} = r_{2m+1}|f'(x_{2m+1})| = 2r_{2m+1}x_{2m+1}$. 
It follows that 
\begin{equation}
\varepsilon_{2m+1} = \frac{2x_{2m+1}}{\sqrt{2m+1}}
> \frac{2x_{2m}}{2m+1} = \varepsilon_{2m} 
\end{equation}
and the proof is complete.
\qed
\end{proof}


\section{Perspectives}

\label{s:persp}

Suppose that $X=\RR$ and
that $f\colon X\to\RR\colon x\mapsto x^2-1$. 
Let $T$ be the subgradient projector associated with $f$ and
assume that $C=X$. 
We chose 100 randomly chosen starting points in
the interval $[1,10^6]$.
In the following table, we record the performance
of the algorithms; 
here $(r_n,\eta_n)$ signals that \eqref{e:ours} was used,
while $\varepsilon_n$ points to \eqref{e:ccp} with
$\eta_n\equiv 1$. Mean and median refer to the number of
iterations until the current iterate was $10^{-6}$ feasible.
\begin{table}[h!] \centering
\begin{tabular}{@{}lrr@{}} \toprule
Algorithm for $x^2-1$ &Mean &Median\\ \midrule
$(r_n,\eta_n)=\big(1/(n+1),1\big)$ &$11.49$ & $13$ \\
$(r_n,\eta_n)=\big(1/(n+1),2\big)$ &$2$ & $2$ \\
$(r_n,\eta_n)=\big(1/\sqrt{n+1},1\big)$ &$10.83$ & $12$ \\
$(r_n,\eta_n)=\big(1/\sqrt{n+1},2\big)$ &$2$ & $2$ \\
$\varepsilon_n=1/(n+1)$ &$11.81$ & $13$ \\
$\varepsilon_n=1/\sqrt{n+1}$ &$12.19$ & $13$ \\
\bottomrule
 \end{tabular}
\end{table}

\noindent
Now let us instead consider  $f\colon X\to\RR\colon x\mapsto 100 x^2-1$.
The corresponding data are in the following table.
\begin{table}[h!] \centering
\begin{tabular}{@{}lrr@{}} \toprule
Algorithm for $100x^2-1$ &Mean &Median\\ \midrule
$(r_n,\eta_n)=\big(1/(n+1),1\big)$ &$13.29$ & $14$ \\
$(r_n,\eta_n)=\big(1/(n+1),2\big)$ &$12$ & $12$ \\
$(r_n,\eta_n)=\big(1/\sqrt{n+1},1\big)$ &$17.52$ & $19$ \\
$(r_n,\eta_n)=\big(1/\sqrt{n+1},2\big)$ &$105$ & $105$ \\
$\varepsilon_n=1/(n+1)$ &$15.27$ & $16$ \\
$\varepsilon_n=1/\sqrt{n+1}$ &$15.76$ & $17$ \\
\bottomrule
 \end{tabular}
\end{table}

\noindent
We observe that the performance of the algorithms
clearly depends on the step lengths $r_n$ and $\varepsilon_n$, on
the relaxation parameter $\eta_n$, and on the underlying
objective function $f$; however, \emph{the precise nature of this
dependence is rather unclear}.
It would thus be
interesting to perform numerical experiments on a wide variety of
problems and parameter choices with the goal to \emph{obtain guidelines
in the choice of algorithms and parameters} for the user. 

Another avenue for future research is to 
\emph{construct a broad framework} that
encompasses the present as well as previous related 
finite convergence results (see references in Section~\ref{s:compare}). 

\section{Conclusions}

\label{s:conc}

We have obtained new and more general finite convergence results
for a class of algorithms based on cutters.
A key tool was Raik's result on
\fejer\ monotone sequences (Fact~\ref{f:Raik}).

\begin{acknowledgements}
The authors thank two anonymous referees for careful reading, 
constructive comments, and for bringing additional references to
our attention. 
The authors also thank 
Jeffrey Pang for helpful discussions and for pointing
out additional references. 
\end{acknowledgements}


\begin{thebibliography}{999}
%
%
\bibitem{bb96} 
Bauschke, H.H., Borwein, J.M.:
On projection algorithms for solving convex feasibility problems,
{SIAM Review}~38, 367--426 (1996)

\bibitem{MOR}
Bauschke, H.H., Combettes, P.L.:
A weak-to-strong convergence principle
for Fej\'er-monotone methods in Hilbert space,
{Mathematics of Operations Research}~26, 248--264 (2001)

\bibitem{BWWX1}
Bauschke, H.H., Wang, C., Wang, X., Xu, J.:
On subgradient projectors,
\texttt{http://arxiv.org/abs/1403.7135v1} (March 2014)

\bibitem{Cegielski}
Cegielski, A.:
{Iterative Methods for Fixed Point Problems in Hilbert
Spaces},
Lecture Notes in Mathematics 2057, Springer-Verlag,
Berlin, Heidelberg, Germany (2012)

\bibitem{CL}
Censor, Y., Lent, A.:
Cyclic subgradient projections,
{Mathematical Programming}~24, 233--235 (1982)

\bibitem{CS}
Censor, Y., Segal, A.:
Sparse string-averaging and split common fixed points. 
In: 
Leizarowitz, A., Mordukhovich, B.S., Shafrir, I., 
Zaslavski, A.J. (eds):
Nonlinear Analysis and Optimization~I
{Contemporary Mathematics}, vol.~513, pp.~125--142 (2010)

\bibitem{CenZen}
Censor, Y., Zenios, S.A.:
{Parallel Optimization},
Oxford University Press (1997)

\bibitem{Comb93}
Combettes, P.L.:
The foundations of set theoretic estimation,
{Proceedings of the IEEE}~81, 182--208 (1993)

\bibitem{Comb97}
Combettes, P.L.:
Convex set theoretic image recovery by extrapolated 
iterations of parallel subgradient projections, 
{IEEE Transactions on Image Processing}~6, 493--506 (1997)

\bibitem{CombLuo}
Combettes, P.L.,  Luo, J.:
An adaptive level set method for
nondifferentiable constrained image recovery,
{IEEE Transactions on Image Processing}~11, 1295--1304 (2002)

\bibitem{Pauwels}
Pauwels, B.: 
Subgradient projection operators, 
\texttt{http://arxiv.org/abs/1403.7237v1} (March 2014)

\bibitem{Poljak}
Polyak, B.T.:
Minimization of unsmooth functionals,
{U.S.S.R.\ Computational Mathematics and Mathematical Physics}~9,
14--29 (1969)  (The original version appeared in 
{Akademija Nauk SSSR. {\v{Z}}urnal Vy{\v{c}}islitel' no\u{\i} 
Matematiki i Matemati\v{c}esko\u{\i}\ Fiziki}~9 (1969), 509--521.)

\bibitem{Polyakbook}
Polyak, B.T.:
{Introduction to Optimization},
Optimization Software, New York, NY, USA (1987)

\bibitem{PolyakHaifa}
Polyak, B.T.:
Random algorithms for solving convex inequalities.
In 
Butnariu, D., Censor, Y., Reich, S. (eds.):
{Inherently Parallel Algorithms in Feasibility
and Optimization and their Applications},
pp.\ 409--422, Elsevier Science Publishers, Amsterdam,
The Netherlands (2001)

\bibitem{SY}
Slavakis, K., Yamada, I.:
The adaptive projected subgradient method constrained by
families of quasi-nonexpansive mappings and its application to
online learning,
{SIAM Journal on Optimization}~23, 126--152 (2013)

\bibitem{YO1}
Yamada, I., Ogura, N.: 
Adaptive projected subgradient method for asymptotic minimization of 
sequence of nonnegative convex functions,
{Numerical Functional Analysis and Optimization}~25, 593--617 (2004)


\bibitem{YO2}
Yamada, I., Ogura, N.: 
Hybrid steepest descent method for variational inequality problem over 
the fixed point set of certain quasi-nonexpansive mappings,
{Numerical Functional Analysis and Optimization}~25, 619--655
(2004)

\bibitem{YSY}
Yamada, I., Slavakis, K.,  Yamada, K.:
An efficient robust adaptive filtering algorithm based on parallel 
subgradient projection techniques,
{IEEE Transactions on  Signal Processing}~50, 1091--1101 (2002)

\bibitem{YY}
Yamagishi, M., Yamada, I.:
A deep monotone approximation operator based on the best
quadratic lower bound of convex functions,
{IEICE Transactions on Fundamentals of Electronics,
Communications and Computer Sciences}~E91--A, 1858--1866 (2008)

\bibitem{Crombez}
Crombez, G.:
Finding common fixed points of a class of paracontractions,
{Acta Mathematica Hungarica}~103, 233--241 (2004)

\bibitem{BC2011}
Bauschke, H.H., Combettes, P.L.:
{Convex Analysis and Monotone Operator Theory in Hilbert Spaces},
Springer (2011)

\bibitem{Thesis}
Bauschke, H.H.:
{Projection Algorithms and Monotone Operators},
PhD thesis, Simon Fraser University, Burnaby, BC, Canada, 
August 1996. 

\bibitem{BK}
Bauschke, H.H., Kruk, S.G.:
Reflection-projection method for convex feasibility problems
with an obtuse cone,
{Journal of Optimization Theory and Applications}~120, 
503--531 (2004)

\bibitem{Raik}
Raik, E.:
A class of iterative methods with Fej\'er-monotone sequences,
{Eesti NSV Teaduste Akadeemia Toimetised.
F\"u\"usika-Matemaatika}~18,  22--26 (1969) 

\bibitem{Rockprox}
Rockafellar, R.T.:
Monotone operators and the proximal point algorithm,
{SIAM Journal on Control and Optimization}~14, 877--898 (1976)


\bibitem{Fuku}
Fukushima, M.:
A finite convergent algorithm for convex inequalities,
{IEEE Transactions on Automatic Control}~27, 1126--1127 (1982)

\bibitem{DPI}
De Pierro, A.R., Iusem, A.N.:
A finitely convergent ``row-action'' method for the convex
feasibility problem,
{Applied Mathematics and Optimization}~17, 225--235  (1988)

\bibitem{CCP}
Censor, Y., Chen, W., Pajoohesh, H.:
Finite convergence of a subgradient projection method with
expanding controls,
{Applied Mathematics and Optimization}~64, 273--285 (2011)


\bibitem{IusMol86}
Iusem, A.N., Moledo, L.:
A finitely convergent method of simultaneous subgradient
projections for the convex feasibility problem, 
{Matem\'atica Aplicada e Computacional}~5, 169--184 (1986)

\bibitem{IusMol87}
Iusem, A.N., Moledo, L.:
On finitely convergent iterative methods for the convex
feasibility problem, 
{Boletim da Sociedade Brasileira de Matem\'atica}~18,
11--18 (1987)

\bibitem{Pang}
Pang, C.H.J.:
Finitely convergent algorithm for nonconvex inequality problems,
\texttt{http://arxiv.org/abs/1405.7280v1} (May 2014)

\end{thebibliography}


\end{document}